\documentclass[9pt,shortpaper,twoside,web]{IEEEtran}
\usepackage{cite}

\ifCLASSINFOpdf
  \usepackage[pdftex]{graphicx}
  \usepackage{epstopdf}
  \usepackage{subfigure}
  \graphicspath{{./}{./figures/}}
\else
\fi
\usepackage{xcolor}

\usepackage{amsmath}
\usepackage{amssymb}

\def\R{\mathbb{R}}

\newcommand{\C} {\mathbb{C}}

\newtheorem{nnassumption}{\bf Assumption}

\newtheorem{nntheorem}{\bf Theorem}
\newenvironment{theorem}{\begin{nntheorem}\it}{\end{nntheorem}}
\newtheorem{nncorollary}{\bf Corollary}

\newtheorem{nndefinition}{\bf Definition}

\newtheorem{nnproposition}{\bf Proposition}

\newtheorem{nnproblem}{\bf Problem}

\newtheorem{nnlemma}{\bf Lemma}

\newtheorem{nnremark}{\bf Remark}
\newenvironment{remark}{\begin{nnremark} \rm }{\hfill \hspace*{1pt}\hfill $\circ$\end{nnremark}}
\newtheorem{nnexample}{\bf Example}

\newenvironment{proof}{{\bf Proof.}}{\hfill \hspace*{1pt}\hfill $\Box$}

\allowdisplaybreaks

\begin{document}
%
\title{Boundary Output Feedback Stabilization of State Delayed Reaction-Diffusion PDEs}
%
%
%

\author{Hugo~Lhachemi, Robert Shorten
\thanks{Hugo Lhachemi is with Universit{\'e} Paris-Saclay, CNRS, CentraleSup{\'e}lec, Laboratoire des signaux et syst{\`e}mes, 91190, Gif-sur-Yvette, France (e-mail: hugo.lhachemi@centralesupelec.fr). Robert Shorten is with Dyson School of Design Engineering, Imperial College London, London, U.K (e-mail: r.shorten@imperial.ac.uk).
}
}

%
%

\markboth{}%
{Lhachemi \MakeLowercase{\textit{et al.}}}
%



\maketitle

\begin{abstract}
This paper studies the boundary output feedback stabilization of general 1-D reaction-diffusion PDEs in the presence of a state delay in the reaction term. The control input applies through a Robin boundary condition while the system output is selected as a either Dirichlet or Neumann boundary trace. The control strategy takes the form of a finite-dimensional observer-based controller with feedback and observer gains that are computed in order to dominate the state delayed term. For any arbitrarily given value of the state delay, we show the exponential stability of the resulting closed-loop system provided the order of the observer is selected large enough. 
\end{abstract}


%
\IEEEpeerreviewmaketitle

\section{Introduction}\label{sec: Introduction}

Reaction-diffusion partial differential equations (PDEs) commonly arise in the modeling of a great variety of phenomena in many fields of research such as biology, chemistry, computer science, population dynamics, and physics. In these models time delays may arise due local phenomena such as inertia~\cite{polyanin2014nonlinear} or natural feedback processes~\cite{gaffney2006gene,allegretto2007stability,su2009hopf}. When considering active control strategies, time-delays may also arise due to the active implementation of feedback loops. In this context, the development of control strategies for the feedback stabilization of reaction-diffusion PDEs with a delay, either in the control input~\cite{katz2021delayed,katz2021sub,krstic2009control,lhachemi2019lmi,lhachemi2020feedback,lhachemi2021robustness,lhachemi2021predictor,qi2020compensation} or in the state~\cite{hashimoto2016stabilization,kang2017boundary,lhachemi2020boundary}, has attracted much attention in the recent years.    

We specifically consider in this paper the topic of boundary feedback stabilization of 1-D reaction-diffusion PDEs in the presence of a state delay. One of the very first contributions regarding the boundary stabilization of such systems was reported in~\cite{hashimoto2016stabilization} for a Dirichlet boundary control using a backstepping transformation~\cite{krstic2008boundary}. Using a backstepping design, the extension to time-varying delays was reported in~\cite{kang2017boundary} for a PDE-ODE cascade. Finally, the case of general Robin boundary control/condition was studied in~\cite{lhachemi2020boundary}. In this latter reference, the control strategy takes the form of a finite-dimensional state feedback obtained by leveraging spectral-reduction methods~\cite{russell1978controllability,coron2004global,coron2006global}. It is worth noting that all the above contributions address the case of a state feedback. An output feedback control strategy was presented in~\cite{wu2020static} for the sole case of bounded control and observation operators. However, the bilinear matrix inequalities-based stability conditions reported therein are only shown to be sufficient. To the best of our knowledge, no systematic control control design strategy was reported for the boundary output feedback stabilization of state delayed reaction-diffusion PDEs with boundary measurements.

The objective of this paper is to achieve the boundary output feedback stabilization of general 1-D reaction-diffusion PDEs in the presence of a state delay and using boundary measurements. This work directly extends the state-feedback setting studied in~\cite{lhachemi2020boundary} to output feedback configurations. The boundary control/conditions take the form of Dirichlet/Neumann/Robin conditions while the system output is obtained as a either Dirichlet or Neumann trace. The control strategy consists of a finite-dimensional observer-based controller~\cite{curtain1982finite,balas1988finite,harkort2011finite}. We leverage the approach reported first in~\cite{katz2020constructive}, and more specifically on the scaling-based procedures described in~\cite{lhachemi2020finite,lhachemi2021nonlinear} that allow to handle Dirichlet/Neumann boundary measurements while performing, for very general 1-D reaction-diffusion PDEs, the control design directly with the actual control input $u$ and not its time-derivative $v = \dot{u}$ (see, e.g., \cite[Sec.~3.3.]{curtain2012introduction} for generalities on boundary control systems). Using Lyapunov's direct method, we propose a finite-dimensional observer-based control strategy that achieves the exponential stabilization of reaction-diffusion PDEs in the presence of a state delay. In this approach, the feedback and observer gains are computed in order to dominate the state delayed term. Surprisingly, the obtained stability conditions, which ensure the exponential decay of the PDE trajectories in $H^1$ norm, are independent of the value of the state delay $h > 0$. Moreover, we show that these constraints are always feasible when selecting the number of modes of the observer to be large enough. Consequently, a distinguishing feature of the proposed control strategy is that the order of the controller is selected independently of the value of the state delay $h$.

The paper is organized as follows. After introducing relevant notations and properties, the control design problem is described in Section~\ref{sec: preliminaries}. The case of a Dirichlet boundary measurement, including control design and main stability result, is addressed in Section~\ref{sec: Dirichlet measurement}. Then the case of a Neumann boundary measurement is reported in Section~\ref{sec: Neumann measurement}. Numerical computations illustrating the main results of this paper are placed in Section~\ref{sec: numerics}. Finally, concluding remarks are formulated in Section~\ref{sec: conclusion}.

\section{Definitions and problem setting}\label{sec: preliminaries}

\subsection{Definitions and properties}

\subsubsection{Notation}

Real spaces $\R^n$ are endowed with the classical Euclidean norm $\Vert\cdot\Vert$. The corresponding induced norms of matrices are also denoted by $\Vert\cdot\Vert$. For any two vectors $X$ and $Y$ of arbitrary dimensions, we denote by $ \mathrm{col} (X,Y)$ the vector $[X^\top,Y^\top]^\top$. The space of square integrable functions on $(0,1)$ is denoted $L^2(0,1)$. This space is endowed with the usual inner product $\langle f , g \rangle = \int_0^1 f(x) g(x) \,\mathrm{d}x$ and the associated norm is denoted by $\Vert \cdot \Vert_{L^2}$. For an integer $m \geq 1$, $H^m(0,1)$ denotes the $m$-order Sobolev space and is equipped with its usual norm $\Vert \cdot \Vert_{H^m}$. For any given symmetric matrix $P \in\R^{n \times n}$, $P \succeq 0$ (resp. $P \succ 0$) means that $P$ is positive semi-definite (resp. positive definite).

\subsubsection{Properties of Sturm-Liouville operators}

Let $\theta_1,\theta_2\in[0,\pi/2]$, $p \in \mathcal{C}^1([0,1])$ and $q \in \mathcal{C}^0([0,1])$ with $p > 0$ and $q \geq 0$. Let the Sturm-Liouville operator $\mathcal{A} : D(\mathcal{A}) \subset L^2(0,1) \rightarrow L^2(0,1)$ be defined by $\mathcal{A}f = - (pf')' + q f$ on the domain $D(\mathcal{A}) = \{ f \in H^2(0,1) \,:\, c_{\theta_1} f(0) - s_{\theta_1} f'(0) = c_{\theta_2} f(1) + s_{\theta_2} f'(1) = 0 \}$ where $c_{\theta_i} = \cos\theta_i$ and $s_{\theta_i} = \sin\theta_i$. The eigenvalues $\lambda_n$, $n \geq 1$, of $\mathcal{A}$ are simple, non negative, and form an increasing sequence with $\lambda_n \rightarrow + \infty$ as $n \rightarrow + \infty$. The corresponding unit eigenvectors $\phi_n \in L^2(0,1)$ form a Hilbert basis. The domain of the operator $\mathcal{A}$ is equivalently characterized by $D(\mathcal{A}) = \{ f \in L^2(0,1) \,:\, \sum_{n\geq 1} \vert \lambda_n \vert ^2 \vert \left< f , \phi_n \right> \vert^2 < +\infty \}$. Let $p_*,p^*,q^* \in \R$ be such that $0 < p_* \leq p(x) \leq p^*$ and $0 \leq q(x) \leq q^*$ for all $x \in [0,1]$, then it holds 
$
0 \leq \pi^2 (n-1)^2 p_* \leq \lambda_n \leq \pi^2 n^2 p^* + q^*
$
for all $n \geq 1$~\cite{orlov2017general}. Moreover if $p \in \mathcal{C}^2([0,1])$, we have (see, e.g., \cite{orlov2017general}) that $\phi_n (\xi) = O(1)$ and $\phi_n' (\xi) = O(\sqrt{\lambda_n})$ as $n \rightarrow + \infty$ for any given $\xi \in [0,1]$. Under the additional assumption $q > 0$, an integration by parts and the continuous embedding $H^1(0,1) \subset L^\infty(0,1)$ show the existence of constants $C_1,C_2 > 0$ such that
\begin{align}
C_1 \Vert f \Vert_{H^1}^2 \leq 
\sum_{n \geq 1} \lambda_n \left< f , \phi_n \right>^2
= \left< \mathcal{A}f , f \right>
\leq C_2 \Vert f \Vert_{H^1}^2 \label{eq: inner product Af and f}
\end{align}
for any $f \in D(\mathcal{A})$. The latter inequalities and the Riesz-spectral nature of $\mathcal{A}$ imply that the series expansion $f = \sum_{n \geq 1} \left< f , \phi_n \right> \phi_n$ holds in $H^2(0,1)$ norm for any $f \in D(\mathcal{A})$. Due to the continuous embedding $H^1(0,1) \subset L^{\infty}(0,1)$, we obtain that $f(0) = \sum_{n \geq 1} \left< f , \phi_n \right> \phi_n(0)$ and $f'(0) = \sum_{n \geq 1} \left< f , \phi_n \right> \phi_n'(0)$. We finally define, for any integer $N \geq 1$, $\mathcal{R}_N f = \sum_{n \geq N+1} \left< f , \phi_n \right> \phi_n$.

\subsection{Problem setting and spectral reduction}

\subsubsection{Problem setting}

We consider the state delayed reaction-diffusion system described by
\begin{subequations}\label{eq: PDE}
\begin{align}
& z_t(t,x) = (p(x) z_x(t,x))_x - \tilde{q}(x) z(t,x) + c z(t-h,x) \\
& c_{\theta_1} z(t,0) - s_{\theta_1} z_x(t,0) = 0 \\
& c_{\theta_2} z(t,1) + s_{\theta_2} z_x(t,1) = u(t) \\
& z(\tau,x) = z_0(\tau,x) , \quad \tau \in [-h,0]
\end{align}
\end{subequations}
for $t > 0$ and $x \in (0,1)$ where $\theta_1 , \theta_2 \in [0,\pi/2]$, $p \in\mathcal{C}^2([0,1])$ with $p > 0$, $\tilde{q} \in\mathcal{C}^0([0,1])$, $c \in\R$ with $c \neq 0$, and the state delay $h > 0$. Here $z(t,\cdot)$ represents the state of the reaction-diffusion equation at time $t$, $u(t)$ stands for the command, and $z_0(\tau,\cdot)$ is the initial condition which is provided for the non positive times $\tau \in [-h,0]$. The system output is selected as the either Dirichlet trace $y_D(t)$, in which case $\theta_1 \in (0,\pi/2]$, or Neumann trace $y_N(t)$, in which case $\theta_1 \in [0,\pi/2)$, defined by
\begin{equation}\label{eq: system output}
y_D(t) = z(t,0) , \qquad y_N(t) = z_x(t,0) .
\end{equation}

For $\tau\in [-h,0]$, we define $u(\tau) = u_0(\tau) = c_{\theta_2} z_0(\tau,1) + s_{\theta_2} z_{0,x}(\tau,1)$. Without loss of generality, we introduce $q \in\mathcal{C}^0([0,1])$ and $q_c \in\R$ so that
\begin{equation}\label{eq: writting of tilde_q}
\tilde{q}(x) = q(x) - q_c , \quad q(x) > 0  .
\end{equation}

\subsubsection{Spectral reduction}

Defining the change of variable
\begin{equation}\label{eq: change of variable}
w(t,x) = z(t,x) - \frac{x^2}{c_{\theta_2} + 2 s_{\theta_2}} u(t) 
\end{equation}
we infer that (\ref{eq: PDE}) can be rewritten as the following equivalent homogeneous representation: 
\begin{subequations}\label{eq: PDE Dirichlet - homogeneous}
\begin{align}
& v(t) = \dot{u}(t) \\
& w_t(t,x) = (p(x) w_x(t,x))_x - \tilde{q}(x) w(t,x) + c w(t-h,x) \\
& \phantom{w_t(t,x) =} \;  + a(x) u(t) - c b(x) u(t-h) + b(x) v(t) \\ 
& c_{\theta_1} w(t,0) - s_{\theta_1} w_x(t,0) = 0 \\
& c_{\theta_2} w(t,1) + s_{\theta_2} w_x(t,1) = 0 \\
& y(t) = w(t,0) \\
& w(\tau,x) = w_0(\tau,x) , \quad \tau \in [-h,0]
\end{align}
\end{subequations}
where we have defined $a(x) = \frac{1}{c_{\theta_2} + 2 s_{\theta_2}} \{ 2p(x) + 2xp'(x) - x^2 \tilde{q}(x) \}$, $b(x) = -\frac{x^2}{c_{\theta_2} + 2 s_{\theta_2}}$, and $w_0(\tau,x) = z_0(\tau,x) - \frac{x^2}{c_{\theta_2} + 2 s_{\theta_2}} u_0(\tau)$. Introducing now the coefficients of projection defined by $z_n(t) = \left< z(t,\cdot) , \phi_n \right>$, $w_n(t) = \left< w(t,\cdot) , \phi_n \right>$, $a_n = \left< a , \phi_n \right>$, and $b_n = \left< b , \phi_n \right>$, we have in particular that
\begin{equation}\label{eq: link z_n and w_n}
w_n(t) = z_n(t) + b_n u(t), \quad n \geq 1 .
\end{equation}
On one hand, the projection of (\ref{eq: PDE}) into the Hilbert basis $(\phi_n)_{n \geq 1}$ implies that (see, e.g., \cite{lhachemi2019iss})
\begin{equation}\label{eq: dynamics z_n}
\dot{z}_n(t) = (-\lambda_n + q_c) z_n(t) + c z_n(t-h) + \beta_n u(t)
\end{equation}
with $\beta_n = a_n + (-\lambda_n+q_c)b_n = p(1) \{ - c_{\theta_2} \phi_n'(1) + s_{\theta_2} \phi_n(1) \} = O(\sqrt{\lambda_n})$, while, on the other hand, the projection of (\ref{eq: PDE Dirichlet - homogeneous}) gives
\begin{subequations}\label{eq: dynamics w_n}
\begin{align}
\dot{u}(t) & = v(t) \\
\dot{w}_n(t) & = (-\lambda_n + q_c) w_n(t) + c w_n(t-h) \\
& \phantom{=}\; + a_n u(t) - c b_n u(t-h) + b_n v(t) 
\end{align}
\end{subequations}
Finally the system outputs (\ref{eq: system output}) are expressed in terms of the coefficients of projection as follows:
\begin{equation}\label{eq: system output - series}
y_D(t) = \sum_{n \geq 1} w_n(t) \phi_n(0) , \quad
y_N(t) = \sum_{n \geq 1} w_n(t) \phi_n'(0) .
\end{equation}

\begin{remark}
Although the development in this paper concerns the case $\theta_1,\theta_2\in[0,\pi/2]$, this is not a limitation of the approach. The approach readily generalizes to $\theta_1,\theta_2\in[0,\pi)$. To do so, one needs to 1) select $q$ in (\ref{eq: writting of tilde_q}) sufficiently large positive so that (\ref{eq: inner product Af and f}) still holds; 2) possibly replace the change of variable formula (\ref{eq: change of variable}), in order to avoid a possible division by 0, by $w(t,x) = z(t,x) - \frac{x^\alpha}{c_{\theta_2} + \alpha s_{\theta_2}} u(t)$ for some $\alpha > 1$ selected so that $c_{\theta_2} + \alpha s_{\theta_2} \neq 0$.
\end{remark}

\section{Case of a Dirichlet measurement}\label{sec: Dirichlet measurement}

We consider in this section the PDE (\ref{eq: PDE}) with Dirichlet measurement $y_D(t)$ described by (\ref{eq: system output}) and with $\theta_1 \in (0,\pi/2]$.

\subsection{Control strategy}

Let $N_0 \geq 1$ be fixed such that $-\lambda_{n} + q_c + \vert c \vert < 0$ for all $n \geq N_0 + 1$. Let $N \geq N_0 + 1$ be arbitrarily given and to be specified later. We consider the following control strategy: 
\begin{subequations}\label{eq: controller part 1 - Dirichlet}
\begin{align}
\hat{w}_n(t) & = \hat{z}_n(t) + b_n u(t) \label{eq: controller 1 - Dirichlet} \\
\dot{\hat{z}}_n(t) & = (-\lambda_n+q_c) \hat{z}_n(t) + c \hat{z}_n(t-h) + \beta_n u(t) \label{eq: controller 2 - Dirichlet} \\
& \phantom{=}\; - l_n \left\{ \sum_{k = 1}^N \hat{w}_k(t) \phi_k(0) - y_D(t) \right\}  ,\; 1 \leq n \leq N_0 \nonumber \\
\dot{\hat{z}}_n(t) & = (-\lambda_n+q_c) \hat{z}_n(t) + c \hat{z}_n(t-h) \label{eq: controller 3 - Dirichlet} \\
& \phantom{=}\; + \beta_n u(t) ,\qquad N_0+1 \leq n \leq N \nonumber \\
u(t) & = \sum_{k=1}^{N_0} k_k \hat{z}_k(t) 
\end{align}
\end{subequations}
where $k_k,l_n \in\R$ are the feedback and observer gains, respectively. 

\begin{remark}\label{rem WP1}
The closed-loop system is composed of the PDE (\ref{eq: PDE}), the Dirichlet measurement $y_D(t)$ described by (\ref{eq: system output}), and the controller (\ref{eq: controller part 1 - Dirichlet}). Its well-posedness in terms of classical solutions for initial conditions $z_0 \in \mathcal{C}^0([-h,0];L^2(0,1))$ and $\hat{z}_n(\tau) = \hat{z}_n(0) \in\R$ for $\tau \in [-h,0]$ such that $z_0,u_0$ are Lipschitz continuous and $z_0(\tau,\cdot) \in H^2(0,1)$ with $c_{\theta_1} z_0(0,0) - s_{\theta_1} z_{0,x}(0,0) = 0$ and $u_0(0) = K \hat{Z}^{N_0}(0)$, is a direct consequence of \cite[Thm.~6.3.1 and~6.3.3]{pazy2012semigroups} by using a classical induction argument.

\end{remark}

\subsection{Truncated model for stability analysis}\label{subsec: Dirichlet - truncated model}

We introduce a finite-dimensional model that captures the $N_0$ first modes of the PDE and the dynamics (\ref{eq: controller part 1 - Dirichlet}) of the output feedback controller. To do so, introducing $\hat{Z}^{N_0} = \begin{bmatrix} \hat{z}_1 & \ldots & \hat{z}_{N_0} \end{bmatrix}^\top$ and $K = \begin{bmatrix} k_1 & \ldots & k_{N_0} \end{bmatrix}$, we have $u = K \hat{Z}^{N_0}$. Defining now the error of observation by $e_n = z_n - \hat{z}_n$ for all $1 \leq n \leq N$, we infer from (\ref{eq: controller 1 - Dirichlet}-\ref{eq: controller 2 - Dirichlet}) and using (\ref{eq: link z_n and w_n}) and (\ref{eq: system output - series}) that
\begin{align*}
\dot{\hat{z}}_n 
& = (-\lambda_n + q_c) \hat{z}_n + c \hat{z}_n(t-h) + \beta_n u + l_n \sum_{k=1}^{N} \phi_k(0) e_k + l_n \zeta
\end{align*}
for $1 \leq n \leq N_0$ with $\zeta(t) = \sum_{n \geq N+1} w_n(t) \phi_n(0)$. Hence with $E^{N_0} = \begin{bmatrix} e_1 & \ldots & e_{N_0} \end{bmatrix}^\top$, $\tilde{e}_n = \sqrt{\lambda_n} e_n$, and $\tilde{E}^{N - N_0} = \begin{bmatrix} \tilde{e}_{N_0 +1} & \ldots & \tilde{e}_{N} \end{bmatrix}^\top$, we obtain that
\begin{align}
\dot{\hat{Z}}^{N_0} 
& = ( A_0 + \mathfrak{B}_0 K ) \hat{Z}^{N_0} + c \hat{Z}^{N_0}(t-h) \nonumber \\
& \phantom{=}\; + LC_0 E^{N_0} + L\tilde{C}_1 \tilde{E}^{N-N_0} + L \zeta \label{eq: truncated model - 4 ODEs - 1}
\end{align}
where $A_0 = \mathrm{diag}(-\lambda_{1} + q_c , \ldots , -\lambda_{N_0} + q_c)$, $\mathfrak{B}_0 = \begin{bmatrix} \beta_1 & \ldots & \beta_{N_0} \end{bmatrix}^\top$, $C_0 = \begin{bmatrix} \phi_1(0) & \ldots & \phi_{N_0}(0) \end{bmatrix}$, $\tilde{C}_1 = \begin{bmatrix} \frac{\phi_{N_0 +1}(0)}{\sqrt{\lambda_{N_0 +1}}} & \ldots & \frac{\phi_{N}(0)}{\sqrt{\lambda_{N}}} \end{bmatrix}$, and $L = \begin{bmatrix} l_1 & \ldots & l_{N_0} \end{bmatrix}^\top$. Note that, following~\cite{lhachemi2020finite}, the scaled error $\tilde{e}_n$ has been introduced to ensure that $\Vert \tilde{C}_1 \Vert = O(1)$ as $N \rightarrow +\infty$. Defining also $A_1 = \mathrm{diag}(-\lambda_{N_0+1} + q_c , \ldots , -\lambda_{N} + q_c)$ and using (\ref{eq: dynamics z_n}) and (\ref{eq: controller 2 - Dirichlet}-\ref{eq: controller 3 - Dirichlet}), the error dynamics is given by
\begin{subequations}\label{eq: truncated model - 4 ODEs - 3}
\begin{align}
\dot{E}^{N_0} & = ( A_0 - L C_0 ) E^{N_0} + c E^{N_0}(t-h) \nonumber \\
& \phantom{=}\; - L \tilde{C}_1 \tilde{E}^{N-N_0} - L \zeta \\
\dot{\tilde{E}}^{N-N_0} & = A_1 \tilde{E}^{N-N_0} + c \tilde{E}^{N-N_0}(t-h)
\end{align}
\end{subequations}
Therefore, defining the vectors
\begin{equation}\label{eq: truncated model - def X}
X_1 = \mathrm{col}\left( \hat{Z}^{N_0} , E^{N_0} \right) , \quad X_2 = \tilde{E}^{N-N_0}
\end{equation}
we infer from (\ref{eq: truncated model - 4 ODEs - 1}-\ref{eq: truncated model - 4 ODEs - 3}) that
\begin{subequations}\label{eq: truncated model}
\begin{align}
\dot{X}_1 & = F_1 X_1 + F_2 X_2 + c X_1(t-h) + \mathcal{L} \zeta \label{eq: truncated model - 1} \\
\dot{X}_2 & = F_3 X_2 + c X_2(t-h) \label{eq: truncated model - 2}
\end{align}
\end{subequations}
where
\begin{equation*}
F_1 =
\begin{bmatrix}
A_0 + \mathfrak{B}_0 K & LC_0 \\
0 & A_0 - L C_0
\end{bmatrix} ,\;
F_2 =
\begin{bmatrix} 
L \tilde{C}_1 \\ - L \tilde{C}_1
\end{bmatrix} ,\;
F_3 = A_1
\end{equation*}
while $\mathcal{L} = \mathrm{col}(L,-L)$.

The direct application of the Hautus test~\cite[Thm.~3.1]{zhou1996robust} shows that the pairs $(A_0,\mathfrak{B}_0)$ and $(A_0,C_0)$ satisfy the Kalman condition~\cite{zhou1996robust}. Hence, one can always find gains $K \in\R^{1 \times N_0}$ and $L \in\R^{N_0}$ such that all the eigenvalues $\mu_i \in\C$, $1 \leq i \leq 2N_0$, of $F_1$ satisfy $\operatorname{Re} \mu_i < - \vert c \vert$. Introducing $\tilde{X} = \mathrm{col}\left( X_1 , X_2 , \zeta \right)$, we finally note that  
\begin{equation}\label{eq: derivative v of command input u}
u = K \hat{Z}^{N_0} = \tilde{K} X_1  , \;\;  \quad v = \dot{u} = E \tilde{X} + c \tilde{K} X_1(t-h)
\end{equation}
with $\tilde{K} = \begin{bmatrix} K & 0 \end{bmatrix}$ and $E = \tilde{K} \begin{bmatrix} F_1 & F_2 & \mathcal{L} \end{bmatrix}$.

\subsection{Main result}

\begin{theorem}\label{thm1}
Let $\theta_1 \in (0,\pi/2]$, $\theta_2 \in [0,\pi/2]$, $p \in\mathcal{C}^2([0,1])$ with $p > 0$, $\tilde{q} \in\mathcal{C}^0([0,1])$, and $c\in\R$ with $c \neq 0$. Let $q \in\mathcal{C}^0([0,1])$ and $q_c \in\R$ be such that (\ref{eq: writting of tilde_q}) holds. Let $N_0 \geq 1$ be such that $-\lambda_n + q_c + \vert c \vert < 0$ for all $n \geq N_0 + 1$. Let $K\in\R^{1 \times N_0}$ and $L\in\R^{N_0}$ be such that all the eigenvalues $\mu_i \in\C$, $1 \leq i \leq 2N_0$, of $F_1$ satisfy $\operatorname{Re} \mu_i < - \vert c \vert$. Let $\alpha_1,\alpha_2,\alpha_3,\alpha_4 > 0$ be fixed so that $\mathfrak{c} = 1-\frac{1}{2} \left( \frac{\vert c \vert}{\alpha_1} + \frac{1}{\alpha_2} + \frac{1}{\alpha_3} + \frac{\vert c \vert}{\alpha_4} \right) > 0$. For a given $N \geq N_0 +1$, assume that there exist $P \succ 0$, $Q_1,Q_2 \succeq 0$, and $r_1,r_2,\beta,\gamma > 0$ such that 
\begin{equation}\label{eq: thm1 - constraints}
\Psi \prec 0 , \quad \Theta_1 \prec 0 ,\quad \Theta_2 \prec 0, \quad \Theta_3 < 0, \quad \Theta_4 < 0 
\end{equation}
where 
\begin{subequations}
\begin{align}
\Psi & = \begin{bmatrix} \Psi_1 & P F_2 & P \mathcal{L} \\ \ast & \Psi_2 & 0 \\ \ast & \ast & - \beta \end{bmatrix} + 2\alpha_3\gamma \Vert \mathcal{R}_N b \Vert_{L^2}^2 E^\top E , \label{eq: Dirichlet LMI conditions - Psi} \\
\Psi_1 & = F_1^\top P + P F_1 + \vert c \vert P + Q_1 + \alpha_2\gamma \Vert \mathcal{R}_N a \Vert_{L^2}^2 \tilde{K}^\top \tilde{K} , \\
\Psi_2 & = r_1 ( 2 F_3 + \vert c \vert I ) + Q_2 , \\
\Theta_1 & = \vert c \vert P - Q_1 + \left( 2 \alpha_3 \vert c \vert + \alpha_4 \right) \gamma \vert c \vert \Vert \mathcal{R}_N b \Vert_{L^2}^2 \tilde{K}^\top \tilde{K} , \\
\Theta_2 & = r_1 \vert c \vert I - Q_2 , \\
\Theta_3 & = \gamma \alpha_1 \vert c \vert - r_2 , \label{eq: Dirichlet LMI conditions - Theta_5} \\
\Theta_4 & = 2 \gamma ( - \mathfrak{c} \lambda_{N+1} + q_c ) + \beta M_\phi + \frac{r_2}{\lambda_{N+1}}
\end{align}
\end{subequations}
with $M_\phi = \sum_{n \geq N+1} \frac{\vert \phi_n(0) \vert^2}{\lambda_n} < +\infty$. Then, for any given $h > 0$, there exist constants $\delta,M>0$ such that, for any initial conditions $z_0 \in \mathcal{C}^0([-h,0];L^2(0,1))$ and $\hat{z}_n(\tau) = \hat{z}_n(0) \in\R$ for $\tau \in [-h,0]$ so that $z_0,u_0$ are Lipschitz continuous and $z_0(\tau,\cdot) \in H^2(0,1)$ with $c_{\theta_1} z_0(0,0) - s_{\theta_1} z_{0,x}(0,0) = 0$ and $u_0(0) = K \hat{Z}^{N_0}(0)$, the trajectories of the closed-loop system composed of the PDE (\ref{eq: PDE}), the Dirichlet measurement $y_D(t)$ described by (\ref{eq: system output}), and the controller (\ref{eq: controller part 1 - Dirichlet}) satisfy 
\begin{align}
& \Vert z(t) \Vert_{H^1}^2 + \sum_{n = 1}^{N} \hat{z}_n(t)^2
\leq M e^{-2 \delta t} \left( \sup_{\tau \in [-h,0]} \Vert z_0(\tau) \Vert_{L^2}^2 \right. \nonumber \\
& \hspace{2.25cm} \left. + \Vert z_0(0) \Vert_{H^1}^2 + \Vert u_0 \Vert_{\infty}^2 + \sum_{n=1}^{N} \hat{z}_n(0)^2 \right)  \label{eq: thm1 stability estimate}
\end{align}
for all $t \geq 0$. Moreover, there always exist $R \succ 0$, $Q_1,Q_2 \succeq 0$, and $r_1,r_2,\beta,\gamma > 0$ such that (\ref{eq: thm1 - constraints}) hold provided $N$ is selected large enough.
\end{theorem}

\begin{proof}
In view of (\ref{eq: thm1 - constraints}), and by a continuity argument, let $\kappa > 0$ be such that
\begin{subequations}\label{eq: thm1 - constraints bis}
\begin{align}
\Theta_{1,\kappa} & = \vert c \vert P - e^{-2\kappa h} Q_1 + 2 \alpha_3 \gamma c^2 \Vert \mathcal{R}_N b \Vert_{L^2}^2 \tilde{K}^\top \tilde{K} \preceq 0 , \\
\Theta_{1,\kappa}' & = \Theta_{1,\kappa} + \alpha_4\gamma \vert c \vert \Vert \mathcal{R}_N b \Vert_{L^2}^2 \tilde{K}^\top \tilde{K} \preceq 0 \label{eq: thm1 - constraints bis - 2} , \\
\Theta_{2,\kappa} & = r_1 \vert c \vert I - e^{-2\kappa h} Q_2 \preceq 0 , \\
\Theta_{3,\kappa} & = \gamma \alpha_1 \vert c \vert - r_2 e^{-2\kappa h} \preceq 0 .
\end{align}
\end{subequations}
Consider now the functional defined by $V(t) = \sum_{i = 1}^5 V_i(t)$ where
\begin{subequations}\label{eq: Lyapunov function H1 norm}
\begin{align}
V_1(t) & = X_1(t)^\top P X_1(t) + r_1 X_2(t)^\top X_2(t) , \\
V_2(t) & = \int_{t-h}^t e^{-2 \kappa (t-s)} X_1(s)^\top Q_1 X_1(s) \mathrm{d}s , \\
V_3(t) & = \int_{t-h}^t e^{-2 \kappa (t-s)} X_2(s)^\top Q_2 X_2(s) \mathrm{d}s , \\
V_4(t) & = \gamma \sum_{n \geq N+1} \lambda_n w_n(t)^2 , \\
V_5(t) & = r_2 \int_{t-h}^t e^{-2 \kappa (t-s)} \sum_{n \geq N+1} w_n(s)^2 \mathrm{d}s .
\end{align}
\end{subequations}
We compute and estimate (using essentially Young's inequality) the time derivative of the different functions $V_i$ as follows: 
\begin{align*}
\dot{V}_1 
& = 2 X_1^\top P \left\{ F_1 X_1 + F_2 X_2 + c X_1(t-h) + \mathcal{L} \zeta \right\} \\
& \phantom{=}\; + 2 r_1 X_2^\top \left\{ F_3 X_2 + c X_2(t-h) \right\} \\
& \leq X_1^\top ( F_1^\top P + P F_1 + \vert c \vert P ) X_1 + \vert c \vert X_1(t-h)^\top P X_1(t-h) \\
& \phantom{\leq}\; + r_1 X_2^\top (2 F_3 + \vert c \vert I ) X_2 + r_1 \vert c \vert \Vert X_2(t-h) \Vert^2 \\
& \phantom{\leq}\; + 2 X_1^\top P F_2 X_2 + 2 X_1^\top P \mathcal{L} \zeta ,
\end{align*}
\begin{align*}
\dot{V}_2
& = -2\kappa V_2 + X_1^\top Q_1 X_1 - e^{-2\kappa h} X_1(t-h)^\top Q_1 X_1(t-h) , \\
\dot{V}_3
& = -2\kappa V_3 + X_2^\top Q_2 X_2 - e^{-2\kappa h} X_2(t-h)^\top Q_2 X_2(t-h) ,
\end{align*}
\begin{align*}
& \dot{V}_4
= 2 \gamma \sum_{n \geq N+1} \lambda_n (-\lambda_n + q_c) w_n^2 + 2 \gamma c \sum_{n \geq N+1} \lambda_n w_n w_n(t-h) \\
& \phantom{\dot{V}_4 =}\; + 2\gamma \sum_{n \geq N+1} \lambda_n \left\{ a_n u - c b_n u(t-h) + b_n v \right\} w_n \\
& \leq \sum_{n \geq N+1} \lambda_n \left\{ 2\gamma ( - \mathfrak{c} \lambda_n + q_c ) \right\} w_n^2 + \alpha_1 \gamma \vert c \vert \sum_{n \geq N+1} w_n(t-h)^2 \\
& \phantom{\leq}\; + \alpha_2 \gamma \Vert \mathcal{R}_N a \Vert_{L^2}^2 X_1^\top \tilde{K}^\top \tilde{K} X_1 + 2\alpha_3\gamma \Vert \mathcal{R}_N b \Vert_{L^2}^2 \tilde{X}^\top E^\top E \tilde{X} \\
& \phantom{\leq}\; + 2\alpha_3\gamma c^2 \Vert \mathcal{R}_N b \Vert_{L^2}^2 X_1(t-h)^\top \tilde{K}^\top \tilde{K} X_1(t-h) \\
& \phantom{\leq}\; + \alpha_4\gamma \vert c \vert \Vert \mathcal{R}_N b \Vert_{L^2}^2 u(t-h)^2 ,
\end{align*}
and
\begin{align*}
\dot{V}_5
& = -2\kappa V_5 + r_2 \sum_{n \geq N+1} w_n(t)^2 - r_2 e^{-2\kappa h} \sum_{n \geq N+1} w_n(t-h)^2 .
\end{align*}
Moreover, since $\zeta = \sum_{n \geq N+1} \phi_n(0) w_n$, we also have $\zeta^2 \leq M_\phi \sum_{n \geq N+1} \lambda_n w_n^2$. Combining all the above estimates, we obtain that
\begin{align}
\dot{V}
& \leq \tilde{X}^\top \Psi \tilde{X} + \sum_{n \geq N+1} \lambda_n \Gamma_n w_n^2 + X_1(t-h)^\top \Theta_{1,\kappa} X_1(t-h) \nonumber \\
& \phantom{\leq}\; + X_2(t-h)^\top \Theta_{2,\kappa} X_2(t-h) + \Theta_{3,\kappa} \sum_{n \geq N+1} w_n(t-h)^2 \nonumber \\
& \phantom{\leq}\; - 2\kappa (V_2+V_3+V_5) + \alpha_4\gamma \vert c \vert \Vert \mathcal{R}_N b \Vert_{L^2}^2 u(t-h)^2 \label{eq: thm1 - estimate dotV}
\end{align}
with $\Gamma_n = 2\gamma (-\mathfrak{c} \lambda_n + q_c) + \beta M_\phi + \frac{r_2}{\lambda_n} \leq \Gamma_{N+1} = \Theta_4$ for all $n \geq N+1$ where we have used that $\mathfrak{c} > 0$. Invoking (\ref{eq: thm1 - constraints}) and (\ref{eq: thm1 - constraints bis}), we deduce in the case $t \in (0,h)$ that
\begin{align*}
\dot{V} & \leq \tilde{X}^\top \Psi \tilde{X} - \vert \Theta_4 \vert \sum_{n \geq N+1} \lambda_n w_n^2 - 2\kappa (V_2+V_3+V_5) \\
& \phantom{\leq}\; + \alpha_4 \gamma \vert c \vert \Vert \mathcal{R}_N b \Vert_{L^2}^2 u_0(t-h)^2 .
\end{align*}
Similarly in the case $t > h$, using (\ref{eq: derivative v of command input u}) and based on $\Theta_{1,\kappa}'$ defined by (\ref{eq: thm1 - constraints bis - 2}), we obtain that
\begin{equation*}
\dot{V} \leq \tilde{X}^\top \Psi \tilde{X} - \vert \Theta_4 \vert \sum_{n \geq N+1} \lambda_n w_n^2 - 2\kappa (V_2+V_3+V_5) .
\end{equation*}
Recalling that $\Psi \prec 0$ while $\kappa,\vert\Theta_4 \vert > 0$, we infer the existence of $\tilde{\kappa},c_1 > 0$ such that $\dot{V} \leq - 2 \tilde{\kappa} V + c_1 u_0(t-h)$ for all $t \in (0,h)$ while $\dot{V} \leq - 2 \tilde{\kappa} V$ for all $t > h$. The integration of the two above estimates implies the existence of a constant $c_2 > 0$ so that 
\begin{equation}\label{eq: thm1 estimate V}
V(t) \leq c_2 e^{-2\tilde{\kappa} t} ( V(0) + \Vert u_0 \Vert_{\infty}^2 )
\end{equation}
for all $t \geq 0$. From (\ref{eq: derivative v of command input u}) and the definition of $V$, we also infer that 
\begin{equation}\label{eq: thm1 estimate u}
u(t)^2 \leq c_3 V(t) \leq c_4 e^{-2\tilde{\kappa} t} ( V(0) + \Vert u_0 \Vert_{\infty}^2 )
\end{equation}
for some constants $c_3,c_4 > 0$.

Since the dynamics (\ref{eq: dynamics z_n}) of the modes $z_n$ for $N_0 +1 \leq n \leq N$ is not captured by (\ref{eq: truncated model}) and the functional $V$ (only the dynamics of the observation error $e_n$ is), we need to perform an additional stability analysis by taking advantage of (\ref{eq: thm1 estimate u}). Consider $\mathcal{V}_n = \mathcal{V}_{n,1} + \mathcal{V}_{n,2}$ with $\mathcal{V}_{n,1} = z_n^2$ and $\mathcal{V}_{n,2} = e^{\eta h} \vert c \vert \int_{t-h}^t e^{-2\eta(t-s)} z_n(s)^2 \mathrm{d}s$ where $\eta > 0$ is selected such that $-\lambda_{N_0 +1} + q_c + e^{\eta h} \vert c \vert < 0$. The computation of the time derivative of $\mathcal{V}_n$ along the trajectories of (\ref{eq: dynamics z_n}) gives
\begin{align*}
\dot{\mathcal{V}}_n
& = 2 \left\{ (-\lambda_n + q_c) z_n + c z_n(t-h) + \beta_n u \right\} z_n \\
& \phantom{=}\; - 2\eta \mathcal{V}_{n,2} + e^{\eta h} \vert c \vert z_n^2 - e^{-\eta h} \vert c \vert  z_n(t-h)^2 \\
& \leq  \left\{ 2 \left( -\lambda_n + q_c + e^{\eta h} \vert c \vert \right) + \epsilon \right\} z_n^2 - 2 \eta \mathcal{V}_{n,2} +\frac{\beta_n^2}{\epsilon} u^2 .
\end{align*}
for any $\epsilon > 0$ where we have used Young's inequality. Since $-\lambda_n + q_c + e^{\eta h} \vert c \vert \leq -\lambda_{N_0 +1} + q_c + e^{\eta h} \vert c \vert < 0$ for all $n \geq N_0 + 1$, setting $\epsilon = \lambda_{N_0 +1} - q_c - e^{\eta h} \vert c \vert > 0$, we infer the existence of constants $c_5>0$ and $\delta  \in (0,\tilde{\kappa})$ such that $\dot{\mathcal{V}}_n \leq -2\delta  \mathcal{V}_n + c_5 u^2$ for all $N_0 +1 \leq n \leq N$. From (\ref{eq: thm1 estimate u}), we deduce the existence of a constant $c_6> 0$ so that 
\begin{equation}\label{eq: thm1 estimate Nu}
\mathcal{V}_n(t) \leq c_6 e^{-2 \delta  t} ( \mathcal{V}_n(0) + V(0) + \Vert u_0 \Vert_{\infty}^2 )
\end{equation}
for all $t \geq 0$ and $N_0 +1 \leq n \leq N$.

From the definition of $V$ and $\mathcal{V}_n$ and using (\ref{eq: inner product Af and f}), we obtain from (\ref{eq: thm1 estimate V}-\ref{eq: thm1 estimate Nu}) the existence of constants $c_7,c_8,c_9 > 0$ such that 
\begin{align*}
& u(t)^2 + \sum_{n=1}^{N} z_n(t)^2 + \sum_{n=1}^{N} \hat{z}_n(t)^2 + \sum_{n \geq N+1} \lambda_n w_n(t)^2 \\
& \quad \leq c_7 \left( V(t) + \sum_{k = N_0 + 1}^{N} \mathcal{V}_k(t) \right) \\
& \quad \leq c_8 e^{-2\delta t} \left( V(0) + \Vert u_0 \Vert_{\infty}^2 + \sum_{k = N_0 + 1}^{N} \mathcal{V}_k(0) \right) \\
& \quad \leq c_9 e^{-2\delta t} \left( \sup_{\tau \in [-h,0]} \Vert w_0(\tau) \Vert_{L^2}^2 + \Vert w_0(0) \Vert_{H^1}^2 \right. \\
& \hspace{4cm} \left. + \Vert u_0 \Vert_{\infty}^2 + \sum_{n=1}^{N} \hat{z}_n(0)^2 \right) .
\end{align*}
Using now (\ref{eq: inner product Af and f}) and (\ref{eq: change of variable}), we deduce that (\ref{eq: thm1 stability estimate}) holds true for some constant $M > 0$.

It remains to show that there always exist $P \succ 0$, $Q_1,Q_2 \succeq 0$, and $r_1,r_2,\beta,\gamma > 0$ such that the constraints (\ref{eq: thm1 - constraints}) hold provided $N$ is selected large enough. Since all the eigenvalues of $F_1$ have a real part strictly less that $- \vert c \vert < 0$, the matrix $F_1 + \vert c \vert I$ is Hurwitz. Hence we can define $P \succ 0$ to be the solution to the Lyapunov equation $F_1^\top P + P F_1 + 2 \vert c \vert P = - I$. Let $\delta^* = \lambda_{N_0 +1} - q_c - \vert c \vert > 0$, $\epsilon_1 = 1/(2 \vert c \vert \Vert P \Vert) > 0$, and $\epsilon_2 = \delta^*/ \vert c \vert > 0$. We fix $r_2 > 0$ so that $r_2 > \alpha_1 \vert c \vert$. For any given $N \geq N_0 +1$ we fix $\beta = \sqrt{N} > 0$, $\gamma = 1/N > 0$, and $r_1 = \beta/\delta^* = \sqrt{N}/\delta^* > 0$. We define $Q_1 = (1+\epsilon_1) \vert c \vert \left\{ P + ( 2 \alpha_3 \vert c \vert + \alpha_4 ) \gamma \Vert \mathcal{R}_N b \Vert_{L^2}^2 \tilde{K}^\top \tilde{K} \right\} \succeq 0$ and $Q_2 = (1+\epsilon_2) r_1 \vert c \vert I \succeq 0$. This ensures that $\Theta_1,\Theta_2 \prec 0$ for any $N \geq N_0 + 1$. Since $0 < \gamma \leq 1$, we also have $\Theta_3 < 0$ for any $N \geq N_0 + 1$. Noting that $F_3 \preceq (-\lambda_{N_0 +1} + q_c)I$, we have $\Psi_2 = 2 r_1 \left\{ F_3 + (1+\epsilon_2/2) \vert c \vert I \right\} \preceq 2 r_1 \left\{ -\lambda_{N_0 +1} + q_c + (1+\epsilon_2/2) \vert c \vert \right\} I = - \delta^* r_1 I = - \beta I \prec 0$ for any $N \geq N_0 + 1$. From the definitions of $P$ and $\epsilon_1$, we infer that $\Psi_1 \preceq - \frac{1}{2} I + \gamma \left\{ \alpha_2 \Vert \mathcal{R}_N a \Vert_{L^2}^2 + (1+\epsilon_1) \vert c \vert ( 2 \alpha_3 \vert c \vert + \alpha_4 ) \Vert \mathcal{R}_N b \Vert_{L^2}^2 \right\} \tilde{K}^\top \tilde{K}$  which shows that $\Psi_1 \preceq - \frac{1}{4} I$ for $N$ large enough. Hence, again for $N$ large enough, 
\begin{align*}
\Psi \preceq 
\begin{bmatrix} 
- \frac{1}{4} I & P F_2 & P \mathcal{L} \\
\ast & - \beta I & 0 \\
\ast & \ast & -\beta
\end{bmatrix}
+ 2\alpha_3\gamma \Vert \mathcal{R}_N b \Vert_{L^2}^2 E^\top E .
\end{align*}
Recalling that $\beta = \sqrt{N}$ and $\gamma = 1/N$ while $\Vert F_2 \Vert = O(1)$ and $\Vert E \Vert = O(1)$ as $N \rightarrow + \infty$, the Schur complement implies that $\Psi \prec 0$ for $N$ large enough. Finally, since $\mathfrak{c} > 0$, we have $\Theta_4 \rightarrow - \infty$ when $N \rightarrow + \infty$. We have shown for $N \geq N_0 +1$ large enough the existence of $R \succ 0$, $Q_1,Q_2 \succeq 0$, and $r_1,r_2,\beta,\gamma > 0$ such that (\ref{eq: thm1 - constraints}) hold. This completes the proof.
\end{proof}

\begin{remark}
For any given $N \geq N_0 + 1$, the constraints (\ref{eq: thm1 - constraints}) of Theorem~\ref{thm1} with decision variables $P \succ 0$, $Q_1,Q_2 \succeq 0$, and $r_1,r_2,\beta,\gamma > 0$ take the form of LMIs that are independent of the value of the state delay $h > 0$. Hence, selecting $N \geq N_0 + 1$ independently of $h > 0$ and such that the constraints (\ref{eq: thm1 - constraints}) are feasible, we obtain the exponential stability of the closed-loop system composed of the PDE (\ref{eq: PDE}), the Dirichlet measurement $y_D(t)$ defined in (\ref{eq: system output}), and the controller (\ref{eq: controller part 1 - Dirichlet}) for any $h > 0$. Note however that the constants $M,\delta > 0$ of the exponential stability estimate (\ref{eq: thm1 stability estimate}) do depend on the specific value of the delay $h > 0$.
\end{remark}

\section{Case of a Neumann measurement}\label{sec: Neumann measurement}

We now consider in this section the PDE (\ref{eq: PDE}) with Neumann measurement $y_N(t)$ described by (\ref{eq: system output}) and with $\theta_1 \in [0,\pi/2)$.

\subsection{Control strategy}

Let $N_0 \geq 1$ be fixed such that $-\lambda_{n} + q_c + \vert c \vert < 0$ for all $n \geq N_0 + 1$. Let $N \geq N_0 + 1$ be arbitrarily given and to be specified later. We consider the following control strategy: 
\begin{subequations}\label{eq: controller part 1 - Neumann}
\begin{align}
\hat{w}_n(t) & = \hat{z}_n(t) + b_n u(t) \label{eq: controller 1 - Neumann} \\
\dot{\hat{z}}_n(t) & = (-\lambda_n+q_c) \hat{z}_n(t) + c \hat{z}_n(t-h) + \beta_n u(t) \label{eq: controller 2 - Neumann} \\
& \phantom{=}\; - l_n \left\{ \sum_{k = 1}^N \hat{w}_k(t) \phi_k'(0) - y_N(t) \right\}  ,\; 1 \leq n \leq N_0 \nonumber \\
\dot{\hat{z}}_n(t) & = (-\lambda_n+q_c) \hat{z}_n(t) + c \hat{z}_n(t-h) \label{eq: controller 3 - Neumann} \\
& \phantom{=}\; + \beta_n u(t) ,\qquad N_0+1 \leq n \leq N \nonumber \\
u(t) & = \sum_{k=1}^{N_0} k_k \hat{z}_k(t) 
\end{align}
\end{subequations}
where $k_k,l_n \in\R$ are the feedback and observer gains, respectively. The well-posedness of the resulting closed-loop system follows from the same arguments that the ones of Remark~\ref{rem WP1}.

\subsection{Truncated model for stability analysis}

We proceed exactly as in Subsection~\ref{subsec: Dirichlet - truncated model} but we replace the definitions of $\zeta$, $\tilde{e}_n$, $C_0$, and $\tilde{C}_1$ by the following: $\zeta(t) = \sum_{n \geq N+1} w_n(t) \phi_n'(0)$, $\tilde{e}_n = \lambda_n e_n$, $C_0 = \begin{bmatrix} \phi_{1}'(0) & \ldots & \phi_{N}'(0) \end{bmatrix}$, and $\tilde{C}_1 = \begin{bmatrix} \frac{\phi_{N_0 +1}'(0)}{\lambda_{N_0 +1}} & \ldots & \frac{\phi_{N}'(0)}{\lambda_{N}} \end{bmatrix}$. Here, following~\cite{lhachemi2020finite}, the scaled error $\tilde{e}_n$ has been defined to ensure that $\Vert \tilde{C}_1 \Vert = O(1)$ as $N \rightarrow +\infty$. Introducing the vectors $X_1,X_2$ as in (\ref{eq: truncated model - def X}), we infer that (\ref{eq: truncated model}-\ref{eq: derivative v of command input u}) hold. Finally, the pairs $(A_0,\mathfrak{B}_0)$ and $(A_0,C_0)$ satisfy the Kalman condition~\cite{zhou1996robust}.

\subsection{Main result}

\begin{theorem}\label{thm2}
Let $\theta_1 \in [0,\pi/2)$, $\theta_2 \in [0,\pi/2]$, $p \in\mathcal{C}^2([0,1])$ with $p > 0$, $\tilde{q} \in\mathcal{C}^0([0,1])$, and $c\in\R$ with $c \neq 0$. Let $q \in\mathcal{C}^0([0,1])$ and $q_c \in\R$ be such that (\ref{eq: writting of tilde_q}) holds. Let $N_0 \geq 1$ be such that $-\lambda_n + q_c + \vert c \vert < 0$ for all $n \geq N_0 + 1$. Let $K\in\R^{1 \times N_0}$ and $L\in\R^{N_0}$ be such that all the eigenvalues $\mu_i \in\C$, $1 \leq i \leq 2N_0$, of $F_1$ satisfy $\operatorname{Re} \mu_i < - \vert c \vert$. Let $\alpha_1,\alpha_2,\alpha_3,\alpha_4 > 0$ be fixed so that $\mathfrak{c} = 1-\frac{1}{2} \left( \frac{\vert c \vert}{\alpha_1} + \frac{1}{\alpha_2} + \frac{1}{\alpha_3} + \frac{\vert c \vert}{\alpha_4} \right) > 0$. For a given $N \geq N_0 +1$, assume that there exist $\epsilon\in(0,1/2]$, $P \succ 0$, $Q_1,Q_2 \succeq 0$, and $r_1,r_2,\beta,\gamma > 0$ such that 
\begin{equation}\label{eq: thm2 - constraints}
\Psi \prec 0 , \quad \Theta_1 \prec 0 ,\quad \Theta_2 \prec 0, \quad \Theta_3 < 0, \quad \Theta_4 < 0, \quad \Theta_5 > 0 
\end{equation}
where $\Psi,\Theta_1,\Theta_2,\Theta_3$ are defined by (\ref{eq: Dirichlet LMI conditions - Psi}-\ref{eq: Dirichlet LMI conditions - Theta_5}) while
\begin{subequations}
\begin{align}
\Theta_4 & = 2 \gamma ( - \mathfrak{c} \lambda_{N+1} + q_c ) + \beta M_\phi(\epsilon) \lambda_{N+1}^{1/2+\epsilon} + \frac{r_2}{\lambda_{N+1}} \\
\Theta_5 & = 2 \gamma \mathfrak{c} - \frac{\beta M_\phi(\epsilon)}{\lambda_{N+1}^{1/2-\epsilon}}
\end{align}
\end{subequations}
with $M_\phi(\epsilon) = \sum_{n \geq N+1} \frac{\vert \phi'_n(0) \vert^2}{\lambda_n^{3/2+\epsilon}} < +\infty$. Then, for any given $h > 0$, there exist constants $\delta,M>0$ such that, for any initial conditions $z_0 \in \mathcal{C}^0([-h,0];L^2(0,1))$ and $\hat{z}_n(\tau) = \hat{z}_n(0) \in\R$ for $\tau \in [-h,0]$ so that $z_0,u_0$ are Lipschitz continuous and $z_0(\tau,\cdot) \in H^2(0,1)$ with $c_{\theta_1} z_0(0,0) - s_{\theta_1} z_{0,x}(0,0) = 0$ and $u_0(0) = K \hat{Z}^{N_0}(0)$, the trajectories of the closed-loop system composed of the PDE (\ref{eq: PDE}), the Neumann measurement $y_N(t)$ described by (\ref{eq: system output}), and the controller (\ref{eq: controller part 1 - Neumann}) satisfy  (\ref{eq: thm1 stability estimate}) for all $t \geq 0$. Moreover, there always exist $\epsilon \in (0,1/2]$, $P \succ 0$, $Q_1,Q_2 \succeq 0$ and $r_1,r_2,\beta,\gamma > 0$ such that (\ref{eq: thm1 - constraints}) hold provided $N$ is selected large enough.
\end{theorem}

\begin{proof}
We proceed as in the proof of Theorem~\ref{thm1} but we replace the estimate of $\zeta$ by the following: $\zeta^2 \leq M_\phi(\epsilon) \sum_{n \geq N+1} \lambda_n^{3/2+\epsilon} w_n^2$, we infer that (\ref{eq: thm1 - estimate dotV}) holds with $\Gamma_n = 2\gamma (-\mathfrak{c} \lambda_n + q_c) + \beta M_\phi(\epsilon) \lambda_{n}^{1/2+\epsilon} + \frac{r_2}{\lambda_n}$. For $n \geq N+1$, $\lambda_n^{1/2+\epsilon} = \lambda_n/\lambda_n^{1/2-\epsilon} \leq \lambda_n/\lambda_{N+1}^{1/2-\epsilon}$ hence $\Gamma_n \leq -\Theta_5 \lambda_n + 2\gamma q_c + \frac{r_2}{\lambda_n} \leq \Gamma_{N+1} = \Theta_4 < 0$ where we have used that $\Theta_5 > 0$. The proof of the stability estimate (\ref{eq: thm1 stability estimate}) follows now the same arguments that the ones of the proof of Theorem~\ref{thm1}. Similarly, one can show the feasibility of the constraints (\ref{eq: thm2 - constraints}) for $N$ large enough by setting $\epsilon = 1/8$, $\beta = N^{1/8}$, and $\gamma = 1/N^{3/16}$. 
\end{proof}

\section{Numerical example}\label{sec: numerics}

Consider the reaction-diffusion system described by (\ref{eq: PDE}) with $p = 1$, $\tilde{q} = -1$, $c = 3$, $\theta_1 = \pi/3$, and $\theta_2 = 0$. In this scenario, the delay-free system ($h=0$) is open-loop unstable. 

Considering first the case of the Dirichlet measurement $y_D(t)$, we set the feedback gain as $K = -2.2316$ while the observer gain is set as $L = 4.7450$. Then the constraints (\ref{eq: thm1 - constraints}) are found feasible for $N = 2$ using \textsc{Matlab} LMI toolbox. Hence Theorem~\ref{thm1} ensures for any arbitrary value of the state delay $h > 0$ the exponential stability of the closed-loop system composed of the plant (\ref{eq: PDE}) and the controller (\ref{eq: controller part 1 - Dirichlet}) for system trajectories evaluated in $H^1$ norm in the sense of (\ref{eq: thm1 stability estimate}). 

Considering now the case of the Neumann measurement $y_N(t)$, we set $K = -1.0149$ and $L = 4.0937$. In this case, the constraints (\ref{eq: thm2 - constraints}) of Theorem~\ref{thm2} are found feasible for $N = 4$, ensuring the exponential stability of the closed-loop system for any $h > 0$.
    
We illustrate the temporal behavior of the closed-loop system with some numerical simulations in the case of a Dirichlet measurement. We consider the delay $h = 1\,\mathrm{s}$ and the initial condition $z_0(\tau,x) = 10 \cos(5\pi(t-1)) x^2 (x-3/4)$. The simulation results are depicted in Fig.~\ref{fig: sim1 CL}, showing the exponential decay to zero of the PDE trajectory and the error of observation. 

Finally, in order to illustrate the impact of the value of state delay $h > 0$ on the exponential decay rate of the system trajectories, we consider $h \in \{0.5,1,2,5,10\}\,\mathrm{s}$ and $z_0(\tau,x) = 10 x^2 (x-3/4)$. The time domain evolution of the $H^1$ norm $\Vert z(t,\cdot) \Vert_{H^1}$ of the PDE trajectories are depicted in Fig~\ref{fig: sim2} with log scale used for the $Y$ axis. As expected, the exponential decay rate decreases as the value of the state delay $h > 0$ increases.

\begin{figure}
     \centering
     	\subfigure[State of the reaction-diffusion system $z(t,x)$]{
		\includegraphics[width=3.5in]{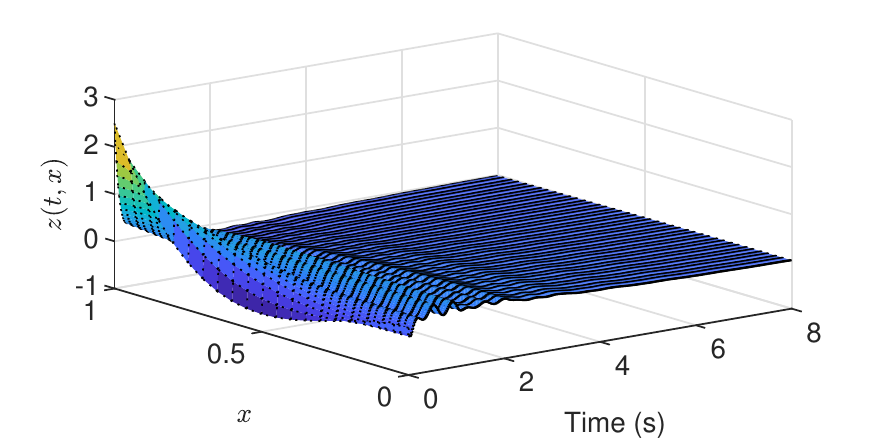}
		}
     	\subfigure[Error of observation $e(t,x) = z(t,x) - \hat{z}(t,x)$]{
		\includegraphics[width=3.5in]{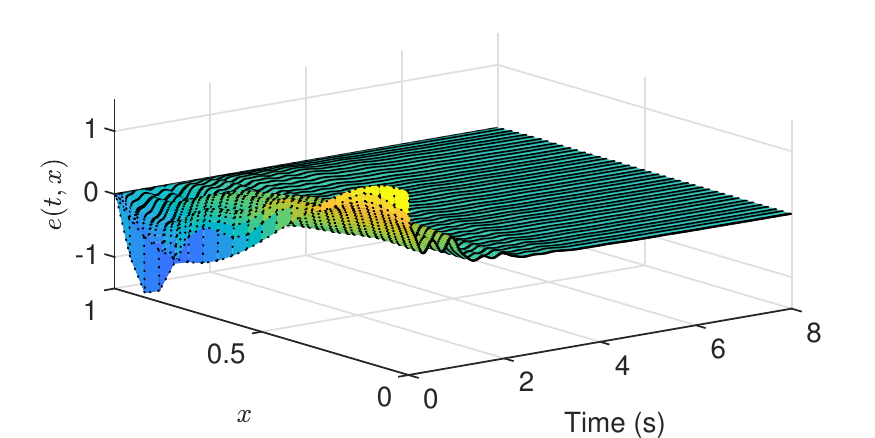}
		}
     \caption{Time evolution of the closed-loop system for Dirichlet measurement with state delay $h = 1\,\mathrm{s}$}
     \label{fig: sim1 CL}
\end{figure}

\begin{figure}
\centering
\includegraphics[width=3.5in]{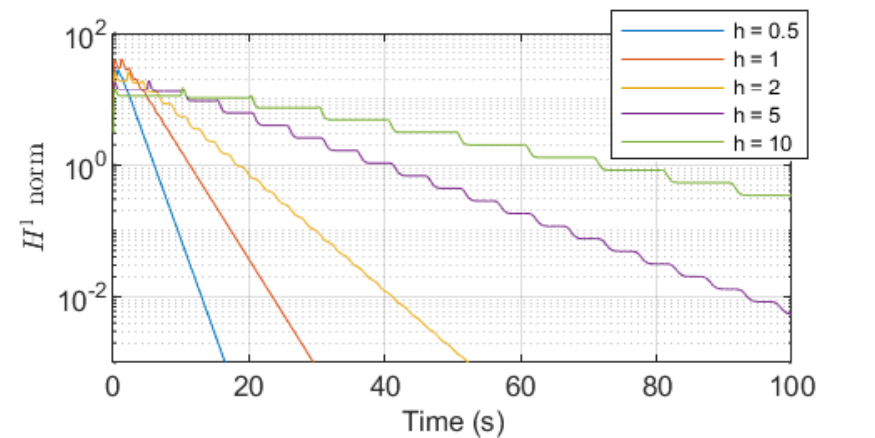}
\caption{Time evolution of the $H^1$ norm $\Vert z(t,\cdot) \Vert_{H^1}$ of the PDE trajectories in closed-loop for Dirichlet measurement with state delays $h \in \{0.5,1,2,5,10\}\,\mathrm{s}$}
\label{fig: sim2}
\end{figure}

\section{Conclusion}\label{sec: conclusion}

This paper solved the problem of boundary output feedback stabilization of state delayed reaction-diffusion PDEs with boundary measurement and using a finite-dimensional controller. We proved that the proposed control strategy achieves the exponential stabilization of the closed-loop system trajectories evaluated in $H^1$ norm for any arbitrarily given value of the state delay $h > 0$. A distinguishing feature is that the order of the controller is selected independently of the value of the state delay $h$.





\ifCLASSOPTIONcaptionsoff
  \newpage
\fi



\bibliographystyle{IEEEtranS}
\nocite{*}
\bibliography{IEEEabrv,mybibfile}

\end{document}